\documentclass[11pt]{article}

\usepackage[marginparsep=1.0in]{geometry}
\usepackage{setspace}
\usepackage{etoolbox}
\AtBeginEnvironment{equation*}{\leavevmode\singlespace}
\AfterEndEnvironment{equation*}{\endsinglespace\vskip0.5\baselineskip\noindent\ignorespaces}

\usepackage{soul} 

\usepackage{xcolor}
\usepackage{enumerate}
\usepackage{tikz}
\usetikzlibrary{decorations.pathreplacing,calligraphy}
\usepackage{float}

\usepackage{hyperref}
\hypersetup{
    colorlinks=true,
    linkcolor=blue,
    filecolor=magenta,      
    urlcolor=blue,
    pdftitle={Overleaf Example},
    pdfpagemode=FullScreen,
    }

\usepackage{amsmath,amssymb,amsthm}

\usepackage{bm}

\numberwithin{equation}{section}

\usepackage{thmtools}
\usepackage{thm-restate}
\usepackage[nameinlink]{cleveref}
\declaretheorem[name=Theorem,numberwithin=section]{theorem}

\newtheorem{proposition}[theorem]{Proposition}
\newtheorem{corollary}{Corollary}[theorem]
\newtheorem{lemma}[theorem]{Lemma}

\theoremstyle{definition}
\newtheorem{definition}[theorem]{Definition}

\theoremstyle{remark}

\newcommand{\C}{\mathbb{C}}

\renewcommand{\S}{\mathcal{S}_{n}}

\newcommand{\ds}{\displaystyle}
\newcommand{\T}{\mathcal{T}_{n}}

\usepackage[symbol]{footmisc}

\usepackage{mathtools}
\DeclarePairedDelimiter{\floor}{\lfloor}{\rfloor}

\title{Structure of singular and nonsingular tournament matrices}
\author{Matt Burnham}
\date{}

\begin{document}

\maketitle

\begin{abstract}
A tournament is a directed graph resulting from an orientation of the complete graph; so, if $M$ is a tournament's adjacency matrix, then $M + M^T$ is a matrix with $0$s on its diagonal and all other entries equal to $1$. An outstanding question in tournament theory asks to classify the adjacency matrices of tournaments which are singular (or nonsingular). We study this question using the structure of tournaments as graphs, in particular their cycle structure. More specifically, we find, as precisely as possible, the number of cycles of length three that dictates whether the corresponding tournament matrix is singular or nonsingular. We also give structural classifications of the tournaments that have the specified numbers of cycles of length three.  
\end{abstract}

\section{Introduction}

A \underline{tournament matrix} of order $n$ is an $n \times n$ $(0,1)$-matrix $M = [m_{ij}]$ which satisfies
$$M + M^T = J_n - I_n,$$
where $J_n$ denotes the $n \times n$ matrix of all 1s and $I_n$ denotes the $n \times n$ identity matrix. A \underline{tournament} of order $n$ is a digraph obtained by arbitrarily orienting each edge of the complete graph on $n$ vertices. Thus a tournament matrix is merely the adjacency matrix of a tournament. Whenever a tournament has property $P$, we say that the corresponding tournament matrix has property $P$ and vice versa. We denote the set of order $n$ tournament matrices by $\T$ and the set of order $n$ singular tournament matrices by $\S$. We denote the number of $3$-cycles in a tournament matrix $M$ by $C_3(M)$.

The \underline{score vector} of an order $n$ tournament matrix $M$ is $s := M\bm{1}_n$ where $\bm{1}_n \in \C^n$ is the vector of 1s. An \underline{ordered score vector} is a score vector in nondecreasing order. In this paper, we refer to an ordered score vector simply as a score vector. From the score vector, we can compute the number of $3$-cycles using the following well-known formula.
\begin{proposition}[Score Vector Formula for $3$-Cycles]
Let $M \in \T$ and $s = (s_1,...,s_n)$ its score vector. The number of three cycles in $M$ is
$$C_3(M) = \binom{n}{3} - \sum_{i=1}^n \binom{s_i}{2}.$$
\end{proposition}

For two vertices $u$ and $v$ in a digraph, if there is an arc from $u$ to $v$, then we say that $u$ beats $v$ and we write $u \to v$. A \underline{transitive tournament} is a tournament such that if $u \to v$ and $v \to w$, then $u \to w$. An order $n$ transitive tournament matrix has the score vector $(0,1,2,...,n-1)$. A \underline{regular tournament} is a tournament whose vertices all have the same outdegree. An order $n$ regular tournament matrix must have odd order and has the score vector $(\frac{n-1}{2},...,\frac{n-1}{2})$. An \underline{almost regular} tournament is a tournament whose maximum difference in outdegrees between vertices is 1. An order $n$ almost regular tournament matrix must have even order and the score vector $(\frac{n}{2}-1,...,\frac{n}{2}-1,\frac{n}{2},...,\frac{n}{2})$.

A pair of vertices in a digraph are called \underline{strongly connected} if there exists a directed walk from one to the other and vice versa. A digraph is called \underline{strongly connected} or \underline{strong} if each pair of vertices is strongly connected. A digraph is called \underline{Hamiltonian} if it contains a directed cycle that traverses through all the vertices, i.e. if it contains a \underline{Hamiltonian cycle}.

Strongly connected is an equivalence relation on the set of vertices in a digraph. Thus a tournament may be partitioned into strongly connected components. The \ul{strongly connected component} in a digraph $D$ containing the vertex $v$ is the induced subdigraph which contains all vertices of $D$ which are strongly connected to $v$. Throughout this paper, we denote the strongly connected components of a tournament matrix $M$ by $M(V_1),...,M(V_k)$ where $V_1,...,V_k$ are their vertex sets.

An \underline{upset tournament} is generated by taking a transitive tournament, finding a directed path from the vertex of degree $n-1$ to the vertex of degree 0, and switching the directions of the arcs on that path. An order $n$ upset tournament has the score vector $(1,1,2,...,n-3,n-2,n-2)$. Burzio showed in \cite{Burzio} that the strongly connected tournaments which contain the least number of $3$-cycles are precisely the upset tournaments.

Colloquially, for a fixed $n$, tournament matrices are on a spectrum from `transitive' to `regular'. $C_3$ is a measurement of this spectrum: transitive tournaments have the least number of $3$-cycles, while regular and almost regular tournaments have the most number of $3$-cycles. Transitive tournament matrices are singular. Regular and almost regular tournament matrices are nonsingular for $n \ge 3$. Thus singular tournament matrices attain an interesting maximum of $C_3$ and nonsingular tournament matrices attain an interesting minimum of $C_3$.

In this paper, we investigate the extrema in Figure \ref{fig:Spectrum}.

\begin{figure}[H]
\centering
\begin{tikzpicture}[scale=1.5,>=stealth,
every node/.style={align=center,scale=.8}] 
\draw[black,line width=1pt] (1,0)--(8,0);
\foreach \i in {1, 3, 6, 8}
\draw[black, ultra thick] (\i,.1)--(\i,-.1);
\draw (1,0.2) node[above]{Transitive};
\draw (8,0.2) node[above]{Regular};
\draw (2,-0.1) node[below]{Singular};
\draw (4.5,-0.1) node[below]{Singular/Nonsingular};
\draw (7,-0.1) node[below]{Nonsingular};
\draw (3,0.2) node[above]{Nonsingular \\ Minimum};
\draw (6,0.2) node[above]{Singular \\ Maximum};
\end{tikzpicture}
\caption{Transitive-Regular Spectrum with Extrema}
\label{fig:Spectrum}
\end{figure}

We give structural classifications of nonsingular tournament matrices which minimize $C_3$ and the number of $3$-cycles they contain in \Cref{thm:nonsingularextreme}.

\begin{restatable*}{theorem}{nonsingularextreme}
\label{thm:nonsingularextreme}
Let $n \ge 3$ and let $M \in \T$. If $M \in \T\setminus\S$, then $\ds C_3(M) \ge n - 2\floor*{\frac{n}{3}}$. Furthermore, $M \in \T\setminus\S$ and $\ds C_3(M) = n - 2\floor*{\frac{n}{3}}$ if and only if the strongly connected components of $M$ are upset tournaments and the number of strongly connected components is $\ds\floor*{\frac{n}{3}}$.
\end{restatable*}

Furthermore, we give the number of $3$-cycles of nonsingular tournament matrices which maximize $C_3$ for $n$ even and we give bounds for $n$ odd.

\begin{restatable*}{proposition}{singularextreme}
\label{prop:singularextreme}
If $n$ is even, then 
$$\max_{M \in \S} C_3(M) = \frac{1}{4}\binom{n}{3}.$$ 
If $n$ is odd, then 
$$2\binom{\frac{n+1}{2}}{3} \le \max_{M \in \S} C_3(M) \le \frac{1}{4}\binom{n}{3}.$$
\end{restatable*}

We do not have encompassing structural characterizations of nonsingular tournament matrices which maximize $C_3$. We are cognizant of a trivial subset: tournaments of order $n$ that are obtained by either adding a sink or source to a regular or almost regular tournament of order $n-1$ (see \Cref{thm:TSE}). We prove in \Cref{prop:nontrivialifandonlyifstrong} that singular tournament matrices which maximize $C_3$ are nontrivial if and only if they are strongly connected.

\section{Maximum Number of $3$-Cycles for Singular Tournament Matrices}

The following proposition has been discovered a multitude of times by different authors. We reference Moon \cite{Moon} for a proof and a small list of original discovers.

\begin{proposition}\label{prop:max3cycles}
Every tournament matrix $M \in \T$ satisfies
$$C_3(M) \le \begin{cases} 
\frac{1}{4}\binom{n+1}{3}, & \text{if $n$ is odd}, \\
2\binom{\frac{n}{2}+1}{3}, & \text{if $n$ is even},
\end{cases}$$
and equality holds if and only if $M$ is a regular or almost regular.
\end{proposition}

Next, we investigate the maximum of $C_3$ over singular tournament matrices.

\begin{lemma}\label{lem:maxSing>=maxPrevTourn}
$$\max_{M \in \mathcal{S}_{n+1}} C_3(M) \ge \max_{M \in \T} C_3(M).$$
\end{lemma}
\begin{proof}
Let $M \in \T$ maximize $C_3$ over $\T$. From the tournament $M$, create a new tournament $N \in \mathcal{T}_{n+1}$ which contains all the vertices and arcs of $M$, plus an additional vertex that is a sink (resp. source); a vertex that is beaten by all other vertices (resp. beats all other vertices). It is straightforward to check that $C_3(N) = C_3(M).$ By construction, the row (resp. column) of $N$ corresponding to the sink (resp. source) is a row of 0s (resp. a column of 0s). Therefore, $N \in \mathcal{S}_{n+1}$.
\end{proof}

Shader showed in \cite{Shader} (Corollary 3.2) that if $C_3(M) > \ds\frac{1}{4}\binom{n}{3}$, then $M$ is nonsingular. Taking the contrapostive yields the following lemma.

\begin{lemma}[Shader's Inequality]\label{lem:Max}
If $M \in \S$, then $C_3(M) \le \ds\frac{1}{4}\binom{n}{3}$.
\end{lemma}

The following proposition gives a precise quantity for $\ds\max_{M \in \S}C_3(M)$ when $n$ is even and bounds when $n$ is odd.

\singularextreme

\begin{proof}
If $n$ is even, then by Shader's inequality, 
$$\frac{1}{4}\binom{n}{3} = \max_{M \in \mathcal{T}_{n-1}} C_3(M) \le \max_{M \in \S} C_3(M) \le \frac{1}{4}\binom{n}{3}$$
whence $\ds\max_{M \in \S} C_3(M) = \frac{1}{4}\binom{n}{3}$. If $n$ is odd, then
$$2\binom{\frac{n+1}{2}}{3} = \max_{M \in \mathcal{T}_{n-1}} C_3(M) \le \max_{M \in \S} C_3(M)\le \frac{1}{4}\binom{n}{3}.$$
\end{proof}

Almost nothing is known about the tournament matrices that maximize $C_3$ over $\S$. Some of these matrices are obtained by adding a vertex to a regular or almost regular tournament of order $n-1$ (see the proof of \Cref{lem:maxSing>=maxPrevTourn} and see \Cref{thm:TSE}). The first counterexample is at $n = 7$. From computer computation, there are exactly $3$ nonisomorphic singular tournament matrices which maximize $C_3$ over $\mathcal{S}_7$. These tournament matrices have the score vector $(1,2,2,3,4,4,5)$ and are
\begin{equation*}
    \begin{bmatrix} 
    0 & 1 & 0 & 0 & 0 & 0 & 0 \\
    0 & 0 & 1 & 0 & 0 & 1 & 0 \\
    1 & 0 & 0 & 1 & 0 & 0 & 0 \\
    1 & 1 & 0 & 0 & 1 & 0 & 0 \\
    1 & 1 & 1 & 0 & 0 & 1 & 0 \\
    1 & 0 & 1 & 1 & 0 & 0 & 1 \\
    1 & 1 & 1 & 1 & 1 & 0 & 0 \end{bmatrix}, \qquad \begin{bmatrix} 
    0 & 0 & 1 & 0 & 0 & 0 & 0 \\
    1 & 0 & 0 & 0 & 0 & 1 & 0 \\
    0 & 1 & 0 & 1 & 0 & 0 & 0 \\
    1 & 1 & 0 & 0 & 1 & 0 & 0 \\
    1 & 1 & 1 & 0 & 0 & 1 & 0 \\
    1 & 0 & 1 & 1 & 0 & 0 & 1 \\
    1 & 1 & 1 & 1 & 1 & 0 & 0 \end{bmatrix}, \qquad \begin{bmatrix} 
    0 & 0 & 1 & 0 & 0 & 0 & 0 \\
    1 & 0 & 0 & 1 & 0 & 0 & 0 \\
    0 & 1 & 0 & 0 & 0 & 1 & 0 \\
    1 & 0 & 1 & 0 & 1 & 0 & 0 \\
    1 & 1 & 1 & 0 & 0 & 0 & 1 \\
    1 & 1 & 0 & 1 & 1 & 0 & 0 \\
    1 & 1 & 1 & 1 & 0 & 1 & 0 \end{bmatrix}.
\end{equation*}

\begin{definition}
A singular tournament matrix $M \in \S$ is called a \ul{singular maximizer of $C_3$} if $C_3(M) = \ds \max_{N \in \S} C_3(N)$. A singular maximizer of $C_3$ is called \ul{trivial} if it is obtained by adding a vertex to a regular or almost regular tournament of order $n-1$.
\end{definition}

We classify trivial singular maximizers of $C_3$ in the next proposition; however, we require the following lemma to do so. The following lemma may be proven using the score vector formula for $3$-cycles.

\begin{lemma}
Let $T_1$ be a tournament with score vector $s = (s_1,...,s_n)$ and $v_i \to v_j$. Let $T_2$ be the tournament obtained by reversing the arc between $v_i$ and $v_j$ in $T_1$. Then $C_3(T_2) - C_3(T_1) = s_i - s_j - 1$.
\end{lemma}

\begin{theorem}\label{thm:TSE}
Let $M \in \S$ be a singular maximizer of $C_3$. The following conditions are equivalent:
\begin{enumerate}[\normalfont(i)]
    \item $M$ is a trivial singular maximizer of $C_3$;
    \item $M$ is obtained by adding a source or a sink to a regular or almost regular tournament of order $n-1$;
    \item $M$ has a strongly connected component of order $1$.
\end{enumerate}
\end{theorem}

\begin{proof}
(ii) $\implies$ (iii) is immediate.

(i) $\implies$ (ii) By definition, there exists a regular or almost regular tournament $R$ of order $n-1$, with vertex set $\{v_1,...,v_{n-1}\}$, such that $M = R + v_n$. Let $N$ be the tournament obtained by adding a vertex $v_n$ to $R$ that is a sink; that is, $v_i \to v_n$ for all $i \neq n$. Then $M$ is obtained by reversing some number of arcs which end at $v_n$, let's say $k$, in $N$.  We split up the proof into two cases: when $R$ is regular and when $R$ is almost regular. \\
\underline{Case 1:} $R$ is regular. \\
By the previous lemma,
$$C_3(M) - C_3(N) = \sum_{i=0}^{k-1}\left[\frac{n}{2} - i - 1\right] = \frac{n-k-1}{2}k.$$
Since $R$ is regular, we know that $n$ is even. $N$ is also a singular maximizer of $C_3$, so by \Cref{prop:singularextreme},
$$C_3(M) = C_3(N) + \frac{n-k-1}{2}k = \frac{1}{4}\binom{n}{3} + \frac{n-k-1}{2}k.$$
Note that for $M$ to be a nonsingular minimizer of $C_3$, $\ds\frac{n-k-1}{2}k$ must equal $0$, which only happens when $k = 0$ or $k = n-1$. Therefore the added vertex $v_n$ in $M$ must be a sink or source. \\
\underline{Case 2:} $R$ is almost regular. \\
Let $k = k_1 + k_2$ where $k_1$ of the arcs begin at a vertex with degree $\ds\frac{n+1}{2}$ and $k_2$ with degree $\ds\frac{n-1}{2}$. By the previous lemma,
$$C_3(M) - C_3(N) = \sum_{i=0}^{k_2-1}\left[\frac{n-1}{2} -i - 1\right] + \sum_{i=k_2}^{k-1}\left[\frac{n+1}{2} -i - 1\right] = \frac{n-k}{2}k - k_2.$$
Since $R$ is almost regular, we know that $n$ is odd and hence $\ds C_3(N) = 2\binom{\frac{n+1}{2}}{3}$. By \Cref{prop:singularextreme},
$$C_3(M) - C_3(N) = \frac{n-k}{2}k - k_2 \le \frac{1}{4}\binom{n}{3} - 2\binom{\frac{n+1}{2}}{3} = \frac{n-1}{8}.$$
We know that this inequality holds for $k = 0$ and $k = n-1$; that is, when $v_n$ is a source or a sink. Assume, by way of contradiction, that the inequality holds for some $1 \le k \le n-2$. We know that $\ds k_2 \le \min\left\{k,\frac{n-1}{2}\right\}$. Thus
$$\frac{n-k}{2}k - \min\left\{k,\frac{n-1}{2}\right\} \le \frac{n-k}{2}k - k_2.$$
It is straightforward to show that the function on the LHS is minimized when $k = 1$ and when $k = n-2$, and that its minimum is $\ds\frac{n-3}{2}$. The inequality
$$\frac{n-3}{2} \le \frac{n-k}{2}k - \min\left\{k,\frac{n-1}{2}\right\} \le \frac{n-k}{2}k - k_2 \le \frac{n-1}{8}$$
implies that $n \le 3$. It is easy to check that the proposition holds for $n \le 3$.

(iii) $\implies$ (i) Assume that $M$ has a strongly connected component of order $1$. Let $M(V_1),...,M(V_k)$ be the strongly connected components of $M$. Assume, without loss of generality, that $M(V_k)$ is the strongly connected component of order $1$. Let $N$ be a tournament matrix whose strongly connected components are $M(V_1),...,M(V_{k-1})$. The order of $N$ is $n-1$. By \Cref{prop:singularextreme} and \Cref{prop:max3cycles}, 
$$\begin{cases} 
\frac{1}{4}\binom{n}{3}, & \text{if $n$ is even}, \\
2\binom{\frac{n+1}{2}}{3}, & \text{if $n$ is odd},
\end{cases} \le C_3(M) = C_3(N) \le \begin{cases} 
\frac{1}{4}\binom{n}{3}, & \text{if $n$ is even}, \\
2\binom{\frac{n+1}{2}}{3}, & \text{if $n$ is odd},
\end{cases}.$$
In order for equality to hold, $N$ must be a regular or almost regular tournament of order $n-1$.
\end{proof}

\begin{lemma}
If $M \in \T$ ($n \ge 6$) is not strongly connected and the strongly connected components of $M$ each contain at least $3$ vertices, then 
$$C_3(M) \le \binom{n-2}{3} + 1.$$
\end{lemma}

\begin{proof}
Let $M(V_1),...,M(V_k)$ be the strongly connected components of $M$. Let $N$ be a tournament matrix whose strongly connected components are $M(V_1),...,M(V_{k-1})$. Let $\ell = |V_1| + \cdots + |V_{k-1}|$ be the order of $N$. By \Cref{prop:max3cycles}, 
$$C_3(N) = \sum_{i=1}^{k-1}C_3(M(V_\ell)) \le \begin{cases} 
\frac{1}{4}\binom{\ell+1}{3}, & \text{if $n$ is odd}, \\
2\binom{\frac{\ell}{2}+1}{3}, & \text{if $n$ is even},
\end{cases} \le \frac{1}{4}\binom{\ell+1}{3}.$$
Therefore,
$$C_3(M) = \sum_{i=1}^{k}C_3(M(V_\ell)) \le \frac{1}{4}\binom{\ell+1}{3} + \frac{1}{4}\binom{n-\ell+1}{3} =  \frac{n}{24}(3\ell^2 - 3n\ell + n^2 - 1).$$
The expression $\frac{n}{24}(3\ell^2 - 3n\ell + n^2 - 1)$ is a quadratic polynomial in terms of $\ell$ with a positive leading coefficient. It is easy to see that $\frac{n}{24}(3\ell^2 - 3n\ell + n^2 - 1)$ is maximized at both endpoints: $\ell = 3$ and $\ell = n-3$. Furthermore, its maximum value is $\ds\frac{1}{4}\binom{n-2}{3} + 1$. Therefore,
$$C_3(M) \le \frac{1}{4}\binom{n-2}{3} + 1.$$
\end{proof}

\begin{proposition}\label{prop:nontrivialifandonlyifstrong}
Let $M \in \S$ ($n \ge 7$) be a singular maximizer of $C_3$. $M$ is nontrivial if and only if $M$ is strongly connected.
\end{proposition}

\begin{proof}
Clearly if $M$ is trivial, then $M$ is is not strongly connected. On the other hand, assume, by way of contradiction, that $M$ is nontrivial and $M$ is not strongly connected. The strongly connected components of $M$ each have at least $3$ vertices. By the previous lemma,
$$C_3(M) \le \frac{1}{4}\binom{n-2}{3} + 1.$$
However, is easy to check that,
$$C_3(M) \le \frac{1}{4}\binom{n-2}{3} + 1 < 2\binom{\frac{n}{2}+1}{3} \le C_3(M).$$
Therefore, if $M$ is nontrivial, then $M$ is strongly connected.
\end{proof}

\section{Minimum Number of $3$-Cycles for Nonsingular Tournament Matrices}

The goal of this section is to prove \Cref{thm:nonsingularextreme}. The following proposition may be proven by looking at the condensation of a tournament matrix, but the proof is omitted here.

\begin{proposition}\label{prop:Determinant}
Let $M$ be a tournament matrix and let $M(V_1),...,M(V_k)$ be its strongly connected components. Then $$\det(M) = \det(M(V_1)) \cdots \det(M(V_k)).$$
\end{proposition}

\begin{corollary}\label{cor:SingIffStrCnnCmpts}
A tournament matrix $M$ is singular if and only if one of its strongly connected components is singular.
\end{corollary}

\begin{corollary}\label{cor:nonsingularscc3verts}
If $M$ is a nonsingular tournament matrix, then the strongly connected components of $M$ contain at least three vertices.
\end{corollary}
\begin{proof}
We prove the contrapositive. Assume $M$ has a strongly connected component with less than 3 vertices. Since a strongly connected component with exactly 2 vertices is impossible, $M$ has a strongly connected component with exactly 1 vertex, say $M(V_\ell)$. Since $M(V_\ell) = [0]$ is singular, we get that $M$ is singular, as desired.
\end{proof}

Note that if $M \in \T$ is an upset tournament, then $C_3(M) = n-2$ (this may be computed using the score vector formula for $3$-cycles). Burzio showed in \cite{Burzio} that the strongly connected tournaments which contain the least number of $3$-cycles are precisely the upset tournaments. This fact gives us our next theorem. 

\begin{theorem}\label{thm:upset}
If $M$ is a strongly connected tournament matrix of order $n \ge 3$, then $C_3(M) \ge n-2$ and equality holds if and only if $M$ is an upset tournament.
\end{theorem}

The following theorem gives structural classifications of nonsingular tournament matrices at the nonsingular extreme.

\nonsingularextreme

\begin{proof}
Suppose that $M \in \T\setminus\S$. Let $M(V_1),M(V_2),...,M(V_k)$ be the strongly connected components of $M$ with orders $\ell_1$, $\ell_2$, ..., $\ell_k$, respectively. By \Cref{cor:nonsingularscc3verts}, $\ell_i \ge 3$. By \Cref{thm:upset}, $C_3(M(V_i)) \ge \ell_i-2$. Therefore,
$$C_3(M) = \sum_{i=1}^k C_3(M(V_i)) \ge \sum_{i=1}^k (\ell_i -2) = \sum_{i=1}^k \ell_i - \sum_{i=1}^k 2 = n - 2k.$$ Since $\ell_i \ge 3$, the number of strongly connected components satisfies $k \le \ds\floor*{\frac{n}{3}}$. Therefore,
\begin{equation}\label{eq:inequality}
    C_3(M) \ge n - 2k \ge n - 2\floor*{\frac{n}{3}}.
\end{equation}

For the forwards implication, assume that $M \in \T\setminus\S$ and that equality holds. If not all of the strongly connected components of $M$ are upset tournaments, then $C_3(M) = n - 2\ds\floor*{\frac{n}{3}} > n - 2k$, which contradicts \eqref{eq:inequality}. Thus the strongly connected components are upsets tournaments. \Cref{eq:inequality} forces $C_3(M) = n - 2k = n - 2\ds\floor*{\frac{n}{3}}$ and it follows that $k = \ds\floor*{\frac{n}{3}}$.

For the backwards implication, assume that the strongly connected components of $M$ are upset tournaments and the number of strongly connected components is $k = \ds\floor*{\frac{n}{3}}$. We first show that $M \in \T\setminus\S$. The strongly connected components of $M$ are upset tournaments of order $3$, $4$, or $5$. Thus the strongly connected components of $M$ have score vectors $(1,1,1)$, $(1,1,2,2)$, or $(1,1,2,3,3)$. The only tournament matrix (up to isomorphism) with the score vector $(1,1,1)$ is
\begin{equation*}
F_1 = \begin{bmatrix}
0 & 1 & 0 \\
0 & 0 & 1 \\
1 & 0 & 0
\end{bmatrix}.
\end{equation*}
The only tournament matrix (up to isomorphism) with the score vector $(1,1,2,2)$ is
\begin{equation*}
F_2 = \begin{bmatrix}
0 & 1 & 0 & 0 \\
0 & 0 & 1 & 0 \\
1 & 0 & 0 & 1 \\
1 & 1 & 0 & 0\end{bmatrix}.    
\end{equation*}
The only tournament matrices (up to isomorphism) with the score vector $(1,1,2,3,3)$ are
\begin{equation*}
F_3 = \begin{bmatrix}
0 & 1 & 0 & 0 & 0 \\
0 & 0 & 1 & 0 & 0 \\
1 & 0 & 0 & 1 & 0 \\
1 & 1 & 0 & 0 & 1 \\
1 & 1 & 1 & 0 & 0 \end{bmatrix} \quad \text{and} \quad 
F_4 = \begin{bmatrix}
0 & 1 & 0 & 0 & 0 \\
0 & 0 & 0 & 1 & 0 \\
1 & 1 & 0 & 0 & 0 \\
1 & 0 & 1 & 0 & 1 \\
1 & 1 & 1 & 0 & 0 \end{bmatrix}. 
\end{equation*}
Observe that $\det(F_1) = \det(F_3) = \det(F_4) = 1$ and $\det(F_2) = -1$. Since isomorphic tournament matrices are permutation similar, the determinants of isomorphic tournaments are the same. By \Cref{cor:SingIffStrCnnCmpts}, $M \in \T\setminus\S$. Switching inequalities to equalities in the derivation of \Cref{eq:inequality} yields that $C_3(M) = n - 2\ds\floor*{\frac{n}{3}}.$
\end{proof}

\begin{corollary}
Let $M \in \T\setminus\S$. If $\ds C_3(M) = n - 2\floor*{\frac{n}{3}}$, then 
$$\det(M) = \begin{cases} 
-1, & \text{if } n \equiv 1 \bmod 3 \\
1, & \text{otherwise.} \end{cases}$$
\end{corollary}
\begin{proof}
If $n \equiv 0 \bmod 3$, then the strongly connected components are all isomorphic to $F_1$. Since isomorphic tournament matrices are permutation similar, their determinants are the same. By \Cref{prop:Determinant}, $\det(M) = \det(F_1) \cdots \det(F_1) = 1$. If $n \equiv 1 \bmod 3$, then all but one strongly connected component is isomorphic to $F_1$. The other has 4 vertices and must be isomorphic to $F_2$. Therefore, $\det(M) = \det(F_1) \cdots \det(F_1)\det(F_2) = -1$. If $n \equiv 2 \bmod 3$, then at most two strongly connected components are not isomorphic to $F_1$. Either they both have $4$ vertices or one has $3$ and the other has $5$. In the former case, both are isomorphic to $F_2$ and hence $\det(M) = \det(F_1) \cdots \det(F_1)\det(F_2)\det(F_2) = 1$. In the latter case, one is isomorphic to $F_1$ and the other is either isomorphic to $F_3$ or $F_4$. Nonetheless, $\det(M) = \det(F_1) \cdots \det(F_1)\det(F_3) = \det(F_1) \cdots \det(F_1)\det(F_4) = 1$.
\end{proof}

This corollary shows that if $M$ is nonsingular and minimizes $C_3$, then $M$ also minimizes the determinant magnitude-wise. This corollary also shows that the nonsingular miminizer of $C_3$ are \ul{unimodular} (i.e. have determinant $\pm 1$). Recall that a matrix is \ul{totally unimodular} if every square submatrix has determinant $\pm 1$ or $0$. Singular minimizers of $C_3$ are not totally unimodular, but they are \textit{almost} totally unimodular.

\begin{corollary}
Let $M \in \T\setminus\S$. If $\ds C_3(M) = n - 2\floor*{\frac{n}{3}}$, then every subtournament matrix of $M$ has determinant $\pm 1$ or $0$.
\end{corollary}
\begin{proof}
Let $N$ be a subtournament of $M$. Then $N$ is obtained by deleting a set of vertices in $M$. Thus the strongly connected components of $N$ are subtournaments of the corresponding strongly connected components of $M$. It is easy to check that $F_1$, $F_2$, $F_3$, and $F_4$ are totally unimodular. Therefore the strongly connected components of $N$ have determinants $\pm 1$ or $0$. Therefore $N$ has determinant $\pm 1$ or $0$.
\end{proof}


\begin{thebibliography}{9}

\bibitem{Shader}
Bryan Shader (1992) \textit{On Tournament Matrices} Linear Algebra and its Applications.

\bibitem{Burzio}
Burzio, M. and Demaria, D.C. (1990), Hamiltonian tournaments with the least number of 3-cycles. J. Graph Theory, 14: 663-672. https://doi.org/10.1002/jgt.3190140606

\bibitem{Moon}
Moon, J. (2013). Strong Subtournaments of a Tournament. In \textit{Topics on tournaments} (p. 11). Project Gutenberg. 


\end{thebibliography}
\end{document}